\documentclass{amsart}

\usepackage{amsmath,amsthm,amssymb,enumerate}
\usepackage{amsfonts}
\usepackage{amsxtra}
\usepackage[citecolor=blue,colorlinks=true]{hyperref}
\allowdisplaybreaks
\usepackage{mathrsfs}
\usepackage[left=3.5cm,right=3.5cm,top=4cm,bottom=4cm]{geometry}

\usepackage{graphicx}

\usepackage{amsmath}

\newcommand{\dif}{\mathrm{d}}
\newcommand{\totdif}{\mathrm{D}}
\newcommand{\mf}{\mathscr{F}}
\newcommand{\mr}{\mathbb{R}}
\newcommand{\p}{\mathbb{P}}
\newcommand{\E}{\mathbb{E}}
\newcommand{\e}{\mathbb{E}}
\newcommand{\ind}{\mathbf{1}}

\newcommand{\mt}{\mathbb{T}^d}
\newcommand{\PP}{\mathbb{P}}

\DeclareMathOperator{\supt}{supp}

\DeclareMathOperator{\diver}{div}
\DeclareMathOperator{\tr}{tr}

\newtheorem{theorem}{Theorem}[section]

\newtheorem{definition}[theorem]{Definition}

\newtheorem{proposition}[theorem]{Proposition}

\begin{document}

\title[Quasilinear parabolic SPDes: existence, uniqueness]{Quasilinear parabolic Stochastic Partial Differential equations: existence, uniqueness}

\author{Martina Hofmanov\'a}
\address[M. Hofmanov\'a]{Technical University Berlin, Institute of Mathematics, Stra\ss e des 17. Juni 136, 10623 Berlin, Germany}
\email{hofmanov@math.tu-berlin.de}

\author{Tusheng Zhang}
\address[T. Zhang]{School of Mathematics, University of Manchester, Oxford
Road, Manchester M13 9PL, England, UK}
\email{tusheng.zhang@manchester.ac.uk}

\begin{abstract}
In this paper, we provide a direct approach to  the existence and uniqueness of
strong (in the probabilistic sense) and weak (in the PDE sense) solutions to quasilinear stochastic partial differential equations, which are neither monotone nor locally monotone.
\end{abstract}

\subjclass[2010]{60H15, 60F10, 35R60}
\keywords{Quasilinear stochastic partial differential equations, strong solutions, energy identity}

\date{\today}

\maketitle

\section{Introduction}

We consider a quasilinear parabolic stochastic partial differential equation of the form
\begin{equation}\label{00.1}
\begin{split}
\dif u+\diver(B(u))\,\dif t&=\diver(A(u)\nabla u)\,\dif t+\sigma(u)\,\dif W(t) ,\quad x\in \mt,\, t\in[0,T],\\
u(0)&=u_0,
\end{split}
\end{equation}
where $W$ is a cylindrical Wiener process in $H=L^2(\mt)$ and $\sigma$ a mapping with values in the space of $\gamma$-radonifying operators from $H$ to certain Sobolev spaces. The coefficients $B: \mr\rightarrow \mr^d$ and $A: \mr\rightarrow \mr^{d\times d}$ are nonlinear functions satisfying suitable hypotheses, in particular, the diffusion matrix $A$ is uniformly elliptic. The precise description of the problem setting will
be given in the next section.

 Equations of this type model the phenomenon of convection-diffusion of ideal fluids and therefore arise in a wide variety of important applications, including for instance two or three phase flows in porous media or sedimentation-consolidation processes (for a thorough exposition of this area given from a practical point of view we refer the reader to \cite{petrol} and the references therein). The addition of a stochastic noise to this physical model is fully natural
as it represents external perturbations or a lack of knowledge of certain physical parameters.

Our aim is to establish existence of a unique solution to \eqref{00.1} that is strong in the probabilistic sense and weak in the PDEs sense. That is, we consider solutions that satisfy \eqref{00.1} with a given driving Wiener process and underlying stochastic basis in the sense of distributions.
Recall that from the probabilistic point of view, two concepts of solution are typically considered in the theory of stochastic evolution equations, namely, pathwise (or strong) solutions and martingale (or
weak) solutions. In the former notion the underlying probability space as well as the driving process is fixed in advance while
in the latter case these stochastic elements become part of the solution of the problem.

The existence and uniqueness of a weak pathwise solution, i.e. strong in the probabilistic sense and weak in the PDE sense, was obtained as a by-product in \cite{DHV}. In this work, the authors were concerned with more general equations, namely degenerate parabolic SPDEs, and made use of the so-called kinetic approach and the notion of kinetic solution. The reason is that in such general cases, the degeneracy of the diffusion matrix introduces further difficulties and the classical notion of PDE weak solution does not provide enough information in order to prove uniqueness and furthermore in some cases the equation might not even be well defined in the sense of distributions.
The proof of existence in \cite{DHV} relies on a Yamada-Watanabe-type argument (see e.g. \cite{krylov}, \cite{PR}). Pathwise uniqueness was established in the context of kinetic solutions which in particular implies uniqueness for PDE weak solutions in case they exist. Existence of a martingale solution was established in \cite{DHV} via stochastic compactness method and these two results were combined and existence of a weak martingale solution deduced.
\vskip 0.4cm
In the present paper we put forward a direct (and therefore much simpler) approach towards existence and uniqueness of \eqref{00.1}. First, we prove existence of a pathwise solution to \eqref{00.1}. The proof is based on a suitable approximation procedure: existence of unique approximating solutions is established using the theory of locally monotone operators \cite{LR} and then we show that these approximations satisfy several uniform bounds. In particular we prove that their gradients are uniformly bounded in space-time and have arbitrarily high moments. This part is based on the recent regularity result of \cite{DDMH}. Finally, we are able to show strong convergence of the approximations and the limit is identified with a solution to \eqref{00.1}. The proof of uniqueness is also new and much simpler than that in \cite{DHV}. We believe the methods presented in this paper are also useful for tackling other type of fully nonlinear SPDEs.

To conclude, let us mention several further references where similar problems were studied. In \cite{DS}, analytical methods were used to prove existence and uniqueness of quasilinear parabolic SPDEs with second order operator having a dominating linear part and with gradient dependence in the noise. In the case of monotone coefficients, the literature is quite extensive, see for instance \cite{PR}, \cite{barbu},\cite{barbu2} and the references therein. For locally monotone coefficients, see \cite{LR}. For equations of gradient type see \cite{gess}.

\section{Mathematical framework}

\subsection{Notations}

In this paper, we adopt the following conventions. We work on a finite-time interval $[0,T],\,T>0,$ and consider periodic boundary conditions, that is, $x\in\mt$ where $\mt=[0,1]^d$ denotes the $d$-dimensional torus. $C^1_b$ denotes the space of continuously differentiable functions, not necessarily bounded but having bounded first order derivative. For $r\in[1,\infty]$, $L^r$ are the Lebesgue spaces and the corresponding norm is denoted by $\|\cdot\|_{L^r}$. In order to measure higher regularity of functions (in the space variable) we make use of the Bessel potential spaces $H^{a,r}(\mt)$, $a\in\mr$ and $r\in(1,\infty)$.
Throughout the paper we will mostly work with the $L^2$-scale and so we will write $H^a$ for $H^{a,2}(\mt)$ and $H$ for $L^2(\mt)$. Recall that for all $a\geq0$, the space $H^{a}$ is the usual Sobolev space of order $a$ with the norm
$$\|u\|_{H^{a}}^2=\sum_{|\alpha|\leq a}\int_{\mt}|\totdif^{\alpha}u|^2\,\dif x$$
and that $H^{-a}$ is the topological dual of $H^{a}$.

\subsection{Hypotheses}
\label{hypo}

Let us now introduce the precise setting of \eqref{00.1}.
We assume that the flux function
$$B=(B_1,\dots,B_d):\mr\longrightarrow \mr^d$$
is of class $C^1_b$. The diffusion matrix $A=(A_{ij})_{i,j=1}^d: \mr\rightarrow \mr^{d\times d}$ is of class $C^1_b$, uniformly positive definite and bounded, i.e. $\delta \mathrm I\leq A\leq C \mathrm I$.

Regarding the stochastic term, let $(\Omega,\mf,\mf_t,\p)$ be a stochastic basis with a complete, right-continuous filtration. The driving process $W$ is a cylindrical Wiener process: it admits the following decomposition
\begin{equation}\label{0.1}
W(t)=\sum_{k=1}^{\infty}\beta_k(t)\bar{e}_k,
\end{equation}
where $(\bar{e}_k)_{k\geq 1}$ is some orthonormal system of the Hilbert space $H$, $(\beta_k)_{k\geq 1}$ is a sequence of independent real-valued Brownian motions relative to $(\mf_t)$.
For each $u\in H$ we consider a mapping $\,\sigma(u):H\rightarrow H$ defined by $\sigma(u)\,\bar e_k=\sigma_k(u(\cdot))$, where $\sigma_k(\cdot): \mr\longrightarrow \mr$
are real-valued functions.
In particular, we suppose that $\sigma$ satisfies the usual Lipschitz condition
\begin{equation}\label{lipsch}
\sum_{k=1}^{\infty}|\sigma_k(y_1)-\sigma_k(y_2)|^2\leq C|y_1-y_2|^2.
\end{equation}
This assumption implies in particular that $\sigma$ maps $H$ to $L_2(H,H)$
where $L_2(H,H)$ denotes the collection of Hilbert-Schmidt operators from $H$ to $H$. Thus, given a predictable process $u$ that belongs to $L^2(\Omega,L^2(0,T;H))$, the stochastic integral $t\mapsto\int_0^t \sigma(u)\dif W$ is a well defined process taking values in $H$ (see \cite[Chapter 4]{daprato} for a thorough exposition).
The equation (\ref{00.1}) can be rewritten as
\begin{equation}\label{00.2}
\begin{split}
\dif u+\diver(B(u))\,\dif t&=\diver(A(u)\nabla u)\,\dif t+\sum_{k=1}^{\infty}\sigma_k(u)\,\dif \beta_k(t) ,\quad x\in \mt,\, t\in[0,T],\\
u(0)&=u_0.
\end{split}
\end{equation}

Later on it will be needed to ensure the existence of the stochastic integral in \eqref{00.1} as an $H^{a,r}$-valued process. We recall that the Bessel potential spaces $H^{a,r}$ with $a\geq 0$ and $r\in[2,\infty)$ belong to the class of $2$-smooth Banach spaces and hence they are well suited for the stochastic It\^o integration (see \cite{gamma1}, \cite{gamma2} for the precise construction of the stochastic integral). So, let us denote by $\gamma(H,X)$ the space of the $\gamma$-radonifying operators from $H$ to a $2$-smooth Banach space $X$. We recall that
 $\Psi\in\gamma(H,X)$ if the series
$$\sum_{k\geq 0}\gamma_k\Psi(e_k)$$
converges in $L^2(\widetilde{\Omega},X)$, for any sequence $(\gamma_k)_{k\geq 0}$ of independent Gaussian real-valued random variables on a probability space $(\widetilde{\Omega},\widetilde{\mathcal{F}},\widetilde{\PP})$ and any orthonormal basis $(e_k)_{k\geq 0}$ of $H$. Then,
this space is endowed with the norm
$$\|\Psi\|_{\gamma(K,X)}:=\Bigg{(}\widetilde{\E}\Bigg{\|}\sum_{k= 1}^\infty\gamma_k\Psi(e_k)\Bigg{\|}_X^2\Bigg{)}^{\frac{1}{2}},\qquad \Psi\in \gamma(H,X),$$
(which does not depend on $(\gamma_k)_{k\geq 1}$, nor on $(e_k)_{k\geq 1}$) and is a Banach space.

With this notation in hand, we state our last assumption upon the coefficient $\sigma$. It coincides with the hypothesis $(H_{a,r})$ from \cite{DDMH}: we assume that for all $a<2$ and $r\in[2,\infty)$
\begin{equation}\label{har}
\begin{split}
\|\sigma(u)\|_{\gamma(H,H^{a,r})}&\leq \begin{cases}
							C\big(1+\|u\|_{H^{a,r}}\big),&\quad a\in[0,1],\\
							C\big(1+\|u\|_{H^{a,r}}+\|u\|_{H^{1,ar}}^a\big),&\quad a>1.
							\end{cases}\\
\end{split}
\end{equation}
Detailed discussion of this condition was provided in \cite{DDMH} so let us just make a few comments here. First of all we observe that there is an overlap between the Lipschitz assumption \eqref{lipsch} and the assumption \eqref{har}. For instance, it follows immediately that \eqref{lipsch} implies \eqref{har} for $a=0$ and all $r\in[2,\infty)$. Nevertheless, for the purposes of our proof it proved useful to keep the two assumptions separate and thus we believe that it would not cause any confusion for the reader. Indeed, assumption \eqref{lipsch} is used several times throughout the paper whereas the use of \eqref{har} is somewhat hidden in Theorem \ref{prop:regul} which is an application of the regularity result from \cite{DDMH}.

We end this section with a definition.

\begin{definition}
An $(\mf_t)$-adapted, $H$-valued continuous process $(u(t), t\geq 0)$ is said to be a solution to equation \eqref{00.2} if
\begin{itemize}
\item[(i)] $u\in L^2(\Omega\times [0,T], H^1)$ for any $T>0$,

\item[(ii)] for any $\phi\in C^{\infty}(\mt)$, $t>0$, the following holds almost surely
\end{itemize}
\begin{equation}\label{00.3}
\begin{split}
\langle u(t), \phi\rangle &-\langle u_0,\phi\rangle-\int_0^t\langle B(u(s)), \nabla \phi\rangle\dif s\\
&=-\int_0^t\langle A(u(s))\nabla u(s), \nabla \phi\rangle\dif s+\int_0^t\langle \sigma(u(s))\,\dif W(s), \phi\rangle.
\end{split}
\end{equation}
\end{definition}
Remark that the solution is a weak solution in the sense of PDEs.

\section{Existence and uniqueness}
\setcounter{equation}{0}

To begin with, write
$$
F(u):=-\diver(B(u))+\diver(A(u)\nabla u).
$$
We have the following estimate:
\begin{equation}\label{1.1}
\|F(u)\|_{H^{-1}}\leq C (1+\|u\|_{H^{1}}), \quad u\in H^{1}.
\end{equation}
Indeed, for $v\in H^{1}$, it holds that
\begin{equation*}\label{1.2}
\begin{split}
  |\langle F(u), v\rangle|&=\left |\langle B(u(x)),\nabla v(x)\rangle-\langle A(u(x))\nabla u(x), \nabla v(x)\rangle\right |\\
  &\leq C(1+\|u\|_H)\|v\|_{H^{1}}+C\|u\|_{H^{1}}\|v\|_{H^{1}}.
\end{split}
\end{equation*}
This implies (\ref{1.1}).
\vskip 0.3cm
Let $P_{\varepsilon}, \varepsilon>0$ denote the semigroup on $H$ generated by the Laplacian on $\mt$. Recall that
$$P_{\varepsilon}f(x)=\int_{\mt}P_{\varepsilon}(x,z)f(z)\,\dif z,$$
here $P_{\varepsilon}(x,z)$ stands for the heat kernel, $x,\,z\in \mt$.
 For $\eta>0$, denote by $C^{\eta}(\mt)$ the space of functions that are $\eta$-H\"o{}lder continuous. We will use the following properties of the semigroup in the sequel.
\begin{equation}\label{1.3}
\|P_{\varepsilon}f\|_{L^{\infty}}\leq C_{\varepsilon}\|f\|_H, \quad f\in H.
\end{equation}
\begin{equation}\label{1.3-1}
\left |\int_{\mt}(P_{\varepsilon_1}(x,z)-P_{\varepsilon_2}(x,z))h(z)\,\dif z\right |\leq C\|h\|_{C^{\eta}(\mt)}(\varepsilon_1-\varepsilon_2)^{\alpha_{\eta}},\quad h\in C^{\eta}(\mt),
\end{equation}
for some $\alpha_{\eta}>0$. We refer the reader to \cite{AG} for these two properties. (\ref{1.3-1}) can also be seen through the relation $\int_{\mt}P_{\varepsilon}(x,z)h(z)\,\dif z=\E[h(B_{\varepsilon}^x)]$, where $B_{\varepsilon}^x$ is the Brownian motion on the torus $\mt$.
For $\varepsilon>0$, $u\in H$, set
\begin{equation}\label{1.4-1}
A_{\varepsilon}(u)(x)=P_{\varepsilon}(A(u))(x), \quad x\in \mt,
\end{equation}
here
$$P_{\varepsilon}(A(u))(x)=(P_{\varepsilon}(A_{ij}(u))(x))_{i,j=1}^d.$$
Consider the following stochastic partial differential equation:
\begin{equation}\label{1.4}
\begin{split}
\dif u^{\varepsilon}(t)+\diver(B(u^{\varepsilon}(t)))\,\dif t&=\diver(A_{\varepsilon}(u^{\varepsilon}(t))\nabla u^{\varepsilon}(t))\,\dif t+\sigma(u^{\varepsilon}(t))\,\dif W(t)\\
u^{\varepsilon}(0)&=u_0.
\end{split}
\end{equation}

\begin{theorem}
Let  $u_0\in L^2(\Omega, H) $. Then there exists a unique solution to the quasi-linear SPDE \eqref{1.4} that satisfies the following energy inequality:
\begin{equation}\label{1.5}
\sup_{\varepsilon}\bigg\{\e\sup_{0\leq t\leq
T}\|u^{\varepsilon}(t)\|_{H}^2+ \int_0^T
\e\|u^{\varepsilon}(t)\|_{H^{1}}^2\,\dif t\bigg\}<\infty.
\end{equation}
\end{theorem}
\vskip 0.3cm

\begin{proof}
First we claim that there exists a constant $C$ such that
 \begin{equation}\label{1.6}
 \delta |\xi|^2\leq A_{\varepsilon}(u)(x)\xi\cdot\xi\leq C |\xi|^2 \quad \mbox{for all}\quad \varepsilon>0, \, u\in H^{1},\, x\in \mt,\, \xi\in \mr^d.
 \end{equation}
 By (H.2), one can find a constant $C$ such that
 \begin{equation}\label{1.7}
  \delta |\xi|^2\leq A(y)\xi\cdot\xi\leq C |\xi|^2\quad\mbox{for all} \quad y\in \mr, \,\xi\in \mr^d.
  \end{equation}
Now,
\begin{equation}\label{1.8}
\begin{split}
  A_{\varepsilon}(u)(x)\xi\cdot\xi&
  =\int_{\mt}P_{\varepsilon}(x,z)A(u(z))\xi\cdot \xi\,\dif z.
\end{split}
\end{equation}
Since $\int_{\mt}P_{\varepsilon}(x,z)\,\dif z=1$, (\ref{1.6}) follows from (\ref{1.7}) and (\ref{1.8}).
Set
$$
F_{\varepsilon}(u):=-\diver(B(u))+\diver(A_{\varepsilon}(u)\nabla u), \quad u\in H^{1}.
$$
For $u\in H^{1}$, by (\ref{1.6})  we have
\begin{equation}\label{1.9}
\begin{split}
  \langle F_{\varepsilon}(u), u\rangle&=\langle B(u), \nabla u\rangle- \langle A_{\varepsilon}(u)\nabla u,\nabla u\rangle \\
  &\leq  (C+C\|u\|_H\|u\|_{H^{1}})-\delta \|u\|_{H^{1}}^2\\
  &\leq C+C\|u\|_H^2-\delta_1 \|u\|_{H^{1}}^2
  \end{split}
\end{equation}
for some constant $\delta_1>0$. Moreover,
\begin{equation}\label{1.10}
\begin{split}
  &\langle F_{\varepsilon}(u)-F_{\varepsilon}(v), u-v\rangle\\
  &=\langle B(u)-B(v), \nabla (u-v)\rangle- \langle A_{\varepsilon}(u)\nabla u-A_{\varepsilon}(v)\nabla v,\nabla (u-v)\rangle \\
  &=\langle B(u)-B(v), \nabla (u-v)\rangle- \langle A_{\varepsilon}(u)\nabla (u-v), \nabla (u-v)\rangle  \\
  &\quad-\langle(A_{\varepsilon}(u)-A_{\varepsilon}(v))\nabla v, \nabla (u-v)\rangle\\
  &\leq C+C\|u-v\|_H^2-\delta_1 \|u-v\|_{H^{1}}^2-\langle(A_{\varepsilon}(u)-A_{\varepsilon}(v))\nabla v, \nabla (u-v)\rangle.
\end{split}
  \end{equation}
  By (\ref{1.3}) and the Lipschitz continuity of $A$ we have
\begin{equation}\label{1.11}
\begin{split}
  -\langle (&A_{\varepsilon}(u)-A_{\varepsilon}(v))\nabla v, \nabla (u-v)\rangle\\
  &\leq \|A_{\varepsilon}(u)-A_{\varepsilon}(v)\|_{L^{\infty}(\mt)}\|v\|_{H^{1}}\|u-v\|_{H^{1}}\\
  &\leq \|P_{\varepsilon}[A(u)-A(v)]\|_{L^{\infty}(\mt)}\|v\|_{H^{1}}\|u-v\|_{H^{1}}\\
  &\leq C_{\varepsilon}\|A(u)-A(v)\|_{H}\|v\|_{H^{1}}\|u-v\|_{H^{1}}\\
  &\leq C\|u-v\|_{H}^2\|v\|_{H^{1}}^2+ \delta_2\|u-v\|_{H^{1}}^2,
  \end{split}
  \end{equation}
  for some constant $\delta_2<\delta_1$.

Putting (\ref{1.10}), (\ref{1.11}) together we arrive at
\begin{equation}\label{1.12}
\begin{split}
  &\langle F_{\varepsilon}(u)-F_{\varepsilon}(u), u-v\rangle\\
  &\leq C+C\|u-v\|_H^2+C\|u-v\|_{H}^2\|v\|_{H^{1}}^2- \delta_3\|u-v\|_{H^{1}}^2,
  \end{split}
  \end{equation}
  for some constant $\delta_3>0$.
  (\ref{1.12}) shows that $F_{\varepsilon}$ satisfies the local monotonicity conditions imposed in \cite{LR}. Applying Theorem 1.1 in \cite{LR},  we obtain the existence and uniqueness of the solution $u^{\varepsilon}$. Next we prove the uniform bound in (\ref{1.5}). By Ito's formula,
  \begin{align*}
  \begin{aligned}
  \|u^{\varepsilon}(t)\|_H^2&=\|u_0\|_H^2-2\int_0^t\langle \diver(B(u^{\varepsilon}(s)),u^{\varepsilon}(s)\rangle  \dif s\\
  &\quad+2\int_0^t\langle \diver(A_{\varepsilon}(u^{\varepsilon}(s))\nabla u^{\varepsilon}(s)),u^{\varepsilon}(s)\rangle \dif s\\
  &\quad+2\int_0^t\langle u^{\varepsilon}(s), \sigma(u^{\varepsilon}(s))\,\dif W(s)\rangle +\sum_{k=1}^{\infty}\int_0^t\|\sigma_k(u^{\varepsilon}(s))\|_H^2\dif s.
\end{aligned}
\end{align*}
By (\ref{1.9}) it follows that
\begin{equation}\label{1.14}
\begin{split}
  &\|u^{\varepsilon}(t)\|_H^2+\delta_1 \int_0^t\|u^{\varepsilon}(s)\|_{H^{1}}^2\dif s\\
  &\leq \|u_0\|_H^2+\int_0^t(C+C\|u^{\varepsilon}(s)\|_H^2)\dif s\\
  &+2\int_0^t\langle u^{\varepsilon}(s), \sigma(u^{\varepsilon}(s))\,\dif W(s)\rangle+\sum_{k=1}^{\infty}\int_0^t\|\sigma_k(u^{\varepsilon}(s))\|_{H}^2\dif s.
\end{split}
\end{equation}
Because the constants involved in the above equation are independent of $\varepsilon$, the uniform bound (\ref{1.5}) follows from (\ref{1.14}), Burkholder's inequality and Gronwall's inequality. The proof is complete.
\end{proof}

\begin{proposition}
Let $u_0\in L^p(\Omega, {\mathcal F}_0; L^p(\mt))$ for some $p\in[2,\infty)$. Then the solutions to \eqref{1.4} satisfy the following estimate
\begin{equation}\label{eq:energy222}
\begin{split}
\sup_\varepsilon\e\sup_{0\leq t\leq T}\|u^\varepsilon(t)\|_{L^p}^p<\infty.
\end{split}
\end{equation}

\begin{proof}
This result is obtained by a suitable version of the It\^o formula using similar arguments as in the proof of \eqref{1.5}. For further details we refer the reader to \cite[Proposition 5.1]{DHV}.
\end{proof}
\end{proposition}

As the next step we establish higher regularity of solutions to \eqref{1.4} which holds true uniformly in $\varepsilon$.

\begin{theorem}\label{prop:regul}
Let $u_0\in L^p(\Omega, {\mathcal F}_0; C^{1+l}(\mt))$ for some $l>0$ and all $p\in [2, \infty)$. Then it holds true that
\begin{equation}\label{1.16}
 \sup_{\varepsilon}\e\sup_{0\leq t\leq T}\|\nabla u^{\varepsilon}(t)\|_{L^{\infty}(\mt)}^p<\infty.
 \end{equation}
\end{theorem}

\begin{proof}
The proof is based on Theorem 2.6 and Theorem 2.7 in \cite{DDMH}. Since the setting in \cite{DDMH} is slightly different, let us explain why the same ideas apply here.

First of all, it is easy to observe that since the arguments for Dirichlet boundary conditions are more involved, considering periodic boundary conditions simplifies the proofs and does not cause any additional difficulties. The main difference between our equation \eqref{1.4} and the model problem from \cite{DDMH} is that our second order operator $A_\varepsilon$ is by definition nonlocal. Let us thus repeat the main ideas from \cite{DDMH} and justify each step.

We consider the following auxiliary problem
\begin{align*}
\begin{aligned}
\dif z^\varepsilon&=\Delta z^\varepsilon\,\dif t+\sigma(u^\varepsilon)\,\dif W,\\
z^\varepsilon(0)&=0,
\end{aligned}
\end{align*}
and define $y^\varepsilon=u^\varepsilon-z^\varepsilon$. Than $y^\varepsilon$ solves
\begin{equation*}
\begin{split}
\partial_t y^\varepsilon&=\diver(A_\varepsilon(u^\varepsilon)\nabla y^\varepsilon)+\diver( B(u^\varepsilon))+\diver((A_\varepsilon(u^\varepsilon)-\mathrm{I})\nabla z^\varepsilon),\\
y(0)&=u_0,
\end{split}
\end{equation*}
which is a (pathwise) deterministic linear parabolic PDE.

To establish the first step in the regularity problem, i.e. \cite[Theorem 2.6]{DDMH}, we remark that all the corresponding estimates for the stochastic convolution $z^\varepsilon$ are valid uniformly in $\varepsilon$ due to \eqref{1.5} and \eqref{eq:energy222}. Moreover, the estimates for $y^\varepsilon$ depend on $A_\varepsilon(u^\varepsilon)$ only through the ellipticity and boundedness constants from \eqref{1.6} and therefore is also independent of $\varepsilon$. Consequently, we deduce that there exists $\eta\in(0,1)$ such that for all $p\in[2,\infty)$
\begin{equation}\label{eq:hold}
\sup_\varepsilon\e \|u^\varepsilon\|_{C^\eta([0,T]\times\mt)}^p<\infty.
\end{equation}

We proceed with the next step proven in \cite[Theorem 2.7, case $k=1$]{DDMH}. The same arguments as above apply to the bounds of the stochastic convolution here. However, concerning the estimates of $y^\varepsilon$ one has to be more careful. In view of \cite[Theorem 3.3]{DDMH}  we need to verify that
\begin{equation}\label{eq:assump}
A_\varepsilon(u^\varepsilon),\,B(u^\varepsilon),\,(A_\varepsilon(u^\varepsilon)-\mathrm{I})\nabla z^\varepsilon\in L^p(\Omega;C^{\alpha/2,\alpha}([0,T]\times\mt))
\end{equation}
for some $\alpha\in(0,1)$, where $C^{\alpha/2,\alpha}([0,T]\times\mt)$ denotes the space of functions that are $\alpha$-H\"older continuous with respect to the parabolic distance
$$d((t,x),(s,y))=\max\{|t-s|^{1/2},|x-y|\}.$$
To this end, we observe that if $f\in C^{\alpha/2,\alpha}([0,T]\times\mt)$ then $P_\varepsilon f\in C^{\alpha/2,\alpha}([0,T]\times\mt)$ uniformly in $\varepsilon$. Indeed, since the convolution kernel $P_\varepsilon(x,z)$ depends only on the difference $x-z$, we have
\begin{align*}
\frac{\big|P_\varepsilon f(t,x)-P_\varepsilon f(s,y)\big|}{d((t,x),(s,y))}&=\frac{\big|\int_{\mt}P_\varepsilon(x-z)f(t,z)\dif z-\int_{\mt}P_\varepsilon(y-z) f(s,z)\dif z\big|}{d((t,x),(s,y))}\\
&\leq\int_{\mt}P_\varepsilon(z)\frac{\big|f(t,x-z)- f(s,y-z)\big|}{d((t,x),(s,y))}\dif z\\
&\leq\|f\|_{C^{\alpha/2,\alpha}([0,T]\times\mt)}.
\end{align*}
Therefore due to \eqref{eq:hold} and the Lipschitz continuity of $A$ we conclude that $A_\varepsilon(u^\varepsilon)$ as well as $(A_\varepsilon(u^\varepsilon)-\mathrm{I})\nabla z^\varepsilon$ possess the regularity required in \eqref{eq:assump} uniformly in $\varepsilon$.
The corresponding statement for $B(u^\varepsilon)$ follows immediately since $B$ is Lipschitz. Finally, \cite[Theorem 2.7]{DDMH} applies in particular \eqref{1.16} follows.
\end{proof}

\begin{theorem}
Let $u_0\in L^m(\Omega, {\mathcal F}_0; C^{1+l}(\mt))$ for some $l>0$ and all $m\in [2, \infty)$. Then there exists a unique solution to the quasi-linear SPDE \eqref{00.1} that satisfies the following energy inequality
\begin{equation}\label{1.15}
\e\sup_{0\leq t\leq
T}\|u(t)\|_{H}^2+ \int_0^T
\e\|u(t)\|_{H^{1}}^2\,\dif t<\infty.
\end{equation}
\end{theorem}

\begin{proof}
We first establish the existence.  Let $u^{\varepsilon}$ be the solution to equation (\ref{1.4}). We will show that $u^{\varepsilon}$ converges to a solution to equation (\ref{00.2}). The estimate (\ref{1.5}) implies that there exist a sequence $(u^{\varepsilon_n}, n\geq 1)$, and  a process
$$u\in L^2(\Omega\times[0,T], H^{1})\cap L^2(\Omega, L^{\infty}(0,T, H)),$$
 for which the following holds:

(i) $u^{\varepsilon_n}\rightarrow u$  weakly in
$L^2(\Omega\times[0,T], H^{1})$, hence weakly in $L^2(\Omega\times[0,T], H)$.

(ii) $u^{\varepsilon_n}\rightarrow u$ in $L^2(\Omega, L^{\infty}(0,T, H))$ with respect to the weak star topology,
\vskip 0.3cm
Next we show that $u^{\varepsilon}$ actually converges to $u$ in $L^1(\Omega, H)$ as $\varepsilon\rightarrow 0$. It is sufficient to prove that $u^{\varepsilon_n}$ is a Cauchy
sequence.

Let $0<\varepsilon_1<\varepsilon_2$. By It\^o's formula,
\begin{equation}\label{1.17}
\begin{split}
& \|u^{\varepsilon_1}(t)-u^{\varepsilon_2}(t)\|_H^2\\
&= 2\int_0^t \langle  B(u^{\varepsilon_1}(s))-B(u^{\varepsilon_2}(s)),\nabla ( u^{\varepsilon_1}(s)-u^{\varepsilon_2}(s))\rangle  \dif s\\
&\quad-2 \int_0^t\langle  A_{\varepsilon_1}(u^{\varepsilon_1}(s))\nabla u^{\varepsilon_1}(s)-A_{\varepsilon_2}(u^{\varepsilon_2}(s))\nabla u^{\varepsilon_2}(s),\nabla ( u^{\varepsilon_1}(s)-u^{\varepsilon_2}(s))\rangle  \dif s\\
&\quad+2\int_0^t\langle u^{\varepsilon_1}(s)-u^{\varepsilon_2}(s), (\sigma(u^{\varepsilon_1}(s))-\sigma(u^{\varepsilon_2}(s)))\,\dif W(s)\rangle  \\
&\quad+\sum_{k=1}^{\infty}\int_0^t\|\sigma_k(u^{\varepsilon_1}(s))-\sigma_k(u^{\varepsilon_2}(s))\|_H^2\dif s\\
&:= I_1(t)+I_2(t)+I_3(t)+I_4(t).
\end{split}
\end{equation}
By the Lipschitz continuity of $B$, for any $\delta_1> 0$ there exists a constant $C_1$ such that
\begin{equation}\label{1.18}
I_1(t)\leq C_1\int_0^t \|u^{\varepsilon_1}(s)-u^{\varepsilon_2}(s)\|_H^2\dif s+\delta_1\int_0^t \|u^{\varepsilon_1}(s)-u^{\varepsilon_2}(s)\|_{H^{1}}^2\dif s.
\end{equation}
By (\ref{1.6}), we have
\begin{equation}\label{1.19}
\begin{split}
I_2(t)
&=-2 \int_0^t\langle  A_{\varepsilon_1}(u^{\varepsilon_1}(s))\nabla (u^{\varepsilon_1}(s)-u^{\varepsilon_2}(s)),\nabla ( u^{\varepsilon_1}(s)-u^{\varepsilon_2}(s))\rangle  \dif s\\
&\quad-2 \int_0^t\langle  (A_{\varepsilon_1}(u^{\varepsilon_1}(s))-A_{\varepsilon_2}(u^{\varepsilon_2}(s)))\nabla u^{\varepsilon_2}(s),\nabla ( u^{\varepsilon_1}(s)-u^{\varepsilon_2}(s))\rangle  \dif s\\
&\leq -2\delta \int_0^t\|u^{\varepsilon_1}(s)-u^{\varepsilon_2}(s)\|_{H^{1}}^2\dif s\\
&\quad-2 \int_0^t\langle  (A_{\varepsilon_1}(u^{\varepsilon_1}(s))-A_{\varepsilon_2}(u^{\varepsilon_1}(s)))\nabla u^{\varepsilon_2}(s),\nabla ( u^{\varepsilon_1}(s)-u^{\varepsilon_2}(s))\rangle  \dif s\\
&\quad-2 \int_0^t\langle  (A_{\varepsilon_2}(u^{\varepsilon_1}(s))-A_{\varepsilon_2}(u^{\varepsilon_2}(s)))\nabla u^{\varepsilon_2}(s),\nabla ( u^{\varepsilon_1}(s)-u^{\varepsilon_2}(s))\rangle  \dif s.
\end{split}
 \end{equation}
 Let $\delta_2$ be a small constant to be fixed later. In view of (\ref{1.3-1}) we have
\begin{equation}\label{1.20}
\begin{split}
-2 \int_0^t&\langle  (A_{\varepsilon_1}(u^{\varepsilon_1}(s))-A_{\varepsilon_2}(u^{\varepsilon_1}(s)))\nabla u^{\varepsilon_2}(s),\nabla ( u^{\varepsilon_1}(s)-u^{\varepsilon_2}(s))\rangle  \dif s\\
&\leq \delta_2 \int_0^t\|u^{\varepsilon_1}(s)-u^{\varepsilon_2}(s)\|_{H^{1}}^2\dif s\\
&\quad+ C \int_0^t\int_{\mt}|A_{\varepsilon_1}(u^{\varepsilon_1}(s))(x)
-A_{\varepsilon_2}(u^{\varepsilon_1}(s))(x)|^2|\nabla u^{\varepsilon_2}(s)|^2(x)\,\dif x\dif s\\
&\leq \delta_2 \int_0^t\|u^{\varepsilon_1}(s)-u^{\varepsilon_2}(s)\|_{H^{1}}^2\dif s\\
&\quad+ C \int_0^t\int_{\mt}\bigg|\int_{\mt}(P_{\varepsilon_1}(x,z)-P_{\varepsilon_2}(x,z))
A(u^{\varepsilon_1}(s))(z)\,\dif z\bigg|^2|\nabla u^{\varepsilon_2}(s)|^2(x)\,\dif x\dif s\\
&\leq \delta_2 \int_0^t\|u^{\varepsilon_1}(s)-u^{\varepsilon_2}(s)\|_{H^{1}}^2\dif s\\
&\quad+ C \int_0^t\int_{\mt}\|A(u^{\varepsilon_1}(s))\|_{C^{\eta}(\mt)}^2(\varepsilon_1-\varepsilon_2)^{2\alpha_{\eta}}
  |\nabla u^{\varepsilon_2}(s)|^2(x)\,\dif x\dif s\\
&\leq \delta_2 \int_0^t\|u^{\varepsilon_1}(s)-u^{\varepsilon_2}(s)\|_{H^{1}}^2\dif s+C (\varepsilon_1-\varepsilon_2)^{2\alpha_{\eta}} \int_0^t\|u^{\varepsilon_2}(s)\|_{H^{1}}^2
\big(1+\|u^{\varepsilon_1}(s)\|_{C^{\eta}(\mt)}^2\big)\dif s.
\end{split}
\end{equation}
Due to Lipschitz continuity of $A$, for any positive constant $\delta_3$ we have
\begin{equation}\label{1.21}
\begin{split}
-2 \int_0^t&\langle  (A_{\varepsilon_2}(u^{\varepsilon_1}(s))-A_{\varepsilon_2}(u^{\varepsilon_2}(s)))\nabla u^{\varepsilon_2}(s),\nabla ( u^{\varepsilon_1}(s)-u^{\varepsilon_2}(s))\rangle  \dif s\\
&\leq \delta_3 \int_0^t\|u^{\varepsilon_1}(s)-u^{\varepsilon_2}(s)\|_{H^{1}}^2\dif s\\
&\quad+ C \int_0^t\int_{\mt}|A_{\varepsilon_2}(u^{\varepsilon_1}(s))(x)
-A_{\varepsilon_2}(u^{\varepsilon_2}(s))(x)|^2|\nabla u^{\varepsilon_2}(s)|^2(x)\,\dif x\dif s\\
&\leq \delta_3 \int_0^t\|u^{\varepsilon_1}(s)-u^{\varepsilon_2}(s)\|_{H^{1}}^2\dif s\\
&\quad+ C \int_0^t\|\nabla u^{\varepsilon_2}(s)\|_{L^{\infty}}^2\int_{\mt}|P_{\varepsilon_2}[A(u^{\varepsilon_1}(s))
-A(u^{\varepsilon_2}(s))]
(x)|^2\,\dif x\dif s\\
&\leq\delta_3 \int_0^t\|u^{\varepsilon_1}(s)-u^{\varepsilon_2}(s)\|_{H^{1}}^2\dif s\\
&\quad+ C \int_0^t\|\nabla u^{\varepsilon_2}(s)\|_{L^{\infty}(\mt)}^2\|u^{\varepsilon_1}(s)-u^{\varepsilon_2}(s)\|_{H}^2\dif s.
\end{split}
\end{equation}
Choose $\delta_2$, $\delta_3$ sufficiently small and substitute (\ref{1.21}), (\ref{1.20}) into (\ref{1.19}) to obtain
\begin{equation}\label{1.22}
\begin{split}
 I_2(t)
&\leq -\delta_4 \int_0^t\|u^{\varepsilon_1}(s)-u^{\varepsilon_2}(s)\|_{H^{1}}^2\dif s\\
&\quad+C (\varepsilon_1-\varepsilon_2)^{2\alpha_{\eta}} \int_0^t\|u^{\varepsilon_2}(s)\|_{H^{1}}^2\big(1+\|u^{\varepsilon_1}(s)\|_{C^{\eta}(\mt)}^2\big)\dif s\\
&\quad+ C \int_0^t\|\nabla u^{\varepsilon_2}(s)\|_{L^{\infty}}^2\|u^{\varepsilon_1}(s)-u^{\varepsilon_2}(s)\|_{H}^2\dif s,
\end{split}
 \end{equation}
 for some $\delta_4>0$.
 Choosing $\delta_1<\delta_4$ it follows from (\ref{1.17}), (\ref{1.18}) and (\ref{1.22}) that
\begin{equation}\label{1.23}
\begin{split}
\|u^{\varepsilon_1}(t)-u^{\varepsilon_2}(t)\|_H^2
&\leq -(\delta_4-\delta_1) \int_0^t\|u^{\varepsilon_1}(s)-u^{\varepsilon_2}(s)\|_{H^{1}}^2\dif s\\
&\quad+C\int_0^t\|u^{\varepsilon_1}(s)-u^{\varepsilon_2}(s)\|_{H}^2\dif s\\
&\quad+C (\varepsilon_1-\varepsilon_2)^{2\alpha_{\eta}} \int_0^t\big(1+\|u^{\varepsilon_1}(s)\|_{C^{\eta}(\mt)}^2\big)\|u^{\varepsilon_2}(s)\|_{H^{1}}^2\dif s\\
&\quad+ C \int_0^t\|\nabla u^{\varepsilon_2}(s)\|_{L^{\infty}}^2\|u^{\varepsilon_1}(s)-u^{\varepsilon_2}(s)\|_{H}^2\dif s\\
&\quad+2\int_0^t\langle u^{\varepsilon_1}(s)-u^{\varepsilon_2}(s), (\sigma(u^{\varepsilon_1}(s))-\sigma(u^{\varepsilon_2}(s)))\,\dif W(s)\rangle.
\end{split}
 \end{equation}
 For any $M>0$, define
 $$\tau_M^{1,2}=\inf\{t>0;\; \|\nabla u^{\varepsilon_2}(t)\|_{L^{\infty}}\geq M\quad\mbox{or}\quad \|u^{\varepsilon_1}(t)\|_{C^{\eta}(\mt)}\geq M\}$$
 with the convention $\inf\emptyset=T$. Then $\tau_M^{1,2}$ is an $(\mf_t)$-stopping time.
 Keeping the bound (\ref{1.5}) in mind and replacing $t$ by $t\wedge\tau_M^{1,2}$ in (\ref{1.23}) we deduce that
 \begin{equation}\label{1.24}
 \begin{split}
& \e\|u^{\varepsilon_1}(t\wedge\tau_M^{1,2} )-u^{\varepsilon_2}(t\wedge\tau_M^{1,2})\|_H^2\\
&\leq C\int_0^t\e\|u^{\varepsilon_1}(s\wedge\tau_M^{1,2})-u^{\varepsilon_2}(s\wedge\tau_M^{1,2})\|_{H}^2\dif s\\
&\quad+C_M \E\int_0^{t\wedge\tau_M^{1,2}}\|u^{\varepsilon_2}(s)\|_{H^1}^2\dif s (\varepsilon_1-\varepsilon_2)^{2\alpha_{\eta}}\\
&\quad+ CM^2 \int_0^t \e\|u^{\varepsilon_1}(s\wedge\tau_M^{1,2})-u^{\varepsilon_2}(s\wedge\tau_M^{1,2})\|_{H}^2\dif s.
\end{split}
 \end{equation}
 By the Gronwall's inequality we obtain from (\ref{1.24}) that
 \begin{equation}\label{1.25}
 \begin{split}
\e\|u^{\varepsilon_1}(t\wedge\tau_M^{1,2})-u^{\varepsilon_2}(t\wedge\tau_M^{1,2})\|_H^2
&\leq C_M (\varepsilon_1-\varepsilon_2)^{2\alpha_{\eta}}\exp(CT+CM^2T).
\end{split}
 \end{equation}
By (\ref{1.5}),
 \begin{equation}\label{1.27}
 \begin{split}
&\e\|u^{\varepsilon_1}(t)-u^{\varepsilon_2}(t)\|_H\\
&=\e\big[\|u^{\varepsilon_1}(t\wedge\tau_M^{1,2} )-u^{\varepsilon_2}(t\wedge\tau_M^{1,2})\|\ind_{ \tau_M^{1,2}\geq t}\big]+\e\big[\|u^{\varepsilon_1}(t)-u^{\varepsilon_2}(t)\|_H\ind_{\tau_M^{1,2}<t}\big]\\
&\leq \e\|u^{\varepsilon_1}(t\wedge\tau_M^{1,2} )-u^{\varepsilon_2}(t\wedge\tau_M^{1,2})\|_H+
C\sup_{\varepsilon}(\e\|u^{\varepsilon}(t)\|_H^2)^{\frac{1}{2}}[\p(\tau_M^{1,2}<t)]^{\frac{1}{2}}\\
&\leq \e\|u^{\varepsilon_1}(t\wedge\tau_M^{1,2} )-u^{\varepsilon_2}(t\wedge\tau_M^{1,2})\|_H+
C[\p(\tau_M^{1,2}<t)]^{\frac{1}{2}}.
\end{split}
 \end{equation}
Given any $\eta>0$. In view of (\ref{1.16}), (\ref{eq:hold}) and (\ref{1.5})  we can first choose $M$ such that
\begin{equation}\label{1.28}
\begin{split}
&C[\p(\tau_M^{1,2}<t)]^{\frac{1}{2}}\\
&\leq C\{[\p(\sup_{0\leq s\leq T}\|\nabla u^{\varepsilon_2}(s)\|_{L^{\infty}(\mt)}\geq M)]^{\frac{1}{2}}+\p(\sup_{0\leq s\leq T}\|u^{\varepsilon_1}(s)\|_{C^{\eta}(\mt)}\geq M)]^{\frac{1}{2}}\}\nonumber\\
&\leq \frac{\eta}{2}, \quad \mbox{for all } \quad \varepsilon_2, \varepsilon_1>0.
\end{split}
\end{equation}
Then use (\ref{1.25}) to find $\varepsilon_0$ so that for $\varepsilon_1, \varepsilon_2\leq \varepsilon_0$,
\begin{equation}\label{1.29}
\e\|u^{\varepsilon_1}(t\wedge\tau_M^{1,2} )-u^{\varepsilon_2}(t\wedge\tau_M^{1,2})\|_H\leq \frac{\eta}{2}.
\end{equation}
Because $\eta$ is arbitrary, we conclude from (\ref{1.27}), (\ref{1.28}) and (\ref{1.29}) that
for all $t\in[0,T]$, $u^\varepsilon(t)\rightarrow u(t)$ in $L^1(\Omega,H)$ and according to \eqref{1.5} and Vitali's convergence theorem, we deduce that
$u^\varepsilon\rightarrow u$ in $L^1(\Omega\times[0,T],H)$, which can be further improved using \eqref{eq:energy222} and Vitali's convergence theorem to
 \begin{equation}\label{1.30}
u^\varepsilon\rightarrow u\quad\text{in}\quad L^p(\Omega\times[0,T],H)\quad\forall p\in[1,\infty).
 \end{equation}

Next we show that the limit process $u$ is a solution to equation (\ref{00.2}). To this end, take a test function $\phi\in C^{\infty}(\mt)$ and use equation
(\ref{1.4}) to get
\begin{equation}\label{1.31}
\begin{split}
&\langle u^{\varepsilon}(t),\phi\rangle -\langle u_0,\phi\rangle -\int_0^t\langle  B(u^{\varepsilon}(s)), \nabla \phi\rangle \dif s\\
&=-\int_0^t\langle  A_{\varepsilon}(u^{\varepsilon}(s))\nabla u^{\varepsilon}(s), \nabla \phi\rangle \dif s+\int_0^t\langle \sigma(u^{\varepsilon}(s))\,\dif W(s), \phi\rangle .
\end{split}
\end{equation}
Taking (\ref{1.30}) into account and letting $\varepsilon\rightarrow 0$ in (\ref{1.31}) we obtain
\begin{align*}
&\e|\langle u^\varepsilon(t)-u(t),\phi\rangle|\rightarrow 0,\\
&\e\bigg|\int_0^t\langle B(u^\varepsilon)-B(u),\nabla\phi\rangle\dif s\bigg|\leq\|\nabla\phi\|_{L^\infty}\e\int_0^t\|u^\varepsilon-u\|_H\dif s\rightarrow 0.
\end{align*}
For the stochastic integral, we have
\begin{align*}
\e\bigg|\int_0^t\langle (\sigma(u^\varepsilon)-\sigma(u))\dif W,\phi\rangle\bigg|&\leq C\e\bigg(\int_0^t\sum_{k=1}^\infty\langle \sigma_k(u^\varepsilon)-\sigma_k(u),\phi\rangle^2\dif s\bigg)^{\frac{1}{2}}\\
&\leq C\|\phi\|_{H}\e\bigg(\int_0^t\|u^\varepsilon-u\|_{H}^2\dif s\bigg)^{\frac{1}{2}}\rightarrow0.
\end{align*}
It remains to pass to the limit in the second order term.
Write
\begin{equation}\label{1.34}
\begin{split}
\int_0^t&\langle  A_{\varepsilon}(u^{\varepsilon}(s))\nabla u^{\varepsilon}(s)-A(u(s))\nabla u(s), \nabla \phi\rangle \dif s\\
&=\int_0^t\langle ( A_{\varepsilon}(u^{\varepsilon}(s))-A_{\varepsilon}(u(s)))\nabla u^{\varepsilon}(s), \nabla \phi\rangle \dif s\\
&\quad+\int_0^t\langle  (A_{\varepsilon}(u(s))-A(u(s)))\nabla u^{\varepsilon}(s), \nabla \phi\rangle \dif s\\
&\quad+\int_0^t\langle  A(u(s))(\nabla u^{\varepsilon}(s)-\nabla u(s)), \nabla \phi\rangle \dif s.
\end{split}
\end{equation}
By the contraction property of the semigroup $P_{\varepsilon}$, Lipschitz continuity of $A$, (\ref{1.5}) and (\ref{1.30}), we have
 \begin{equation}\label{1.35}
 \begin{split}
&\e\bigg|\int_0^t\langle ( A_{\varepsilon}(u^{\varepsilon}(s))-A_{\varepsilon}(u(s)))\nabla u^{\varepsilon}(s), \nabla \phi\rangle \dif s\bigg|\\
&\leq \|\nabla \phi\|_{L^{\infty}}\e\int_0^t\int_{\mt}
|P_{\varepsilon}(A(u^{\varepsilon}(s))-A(u(s)))(x)\|\nabla u^{\varepsilon}(s)(x)|\,\dif x\dif s\\
&\leq C\bigg(\e\int_0^t\|A(u^{\varepsilon}(s))-A(u(s))\|_H^2\dif s\bigg)^{\frac{1}{2}}
\bigg(\int_0^t\e\|u^{\varepsilon}(s)\|_{H^{1}}^2\dif s\bigg)^{\frac{1}{2}}\rightarrow 0.
 \end{split}
\end{equation}
By the strong continuity of the semigroup $P_{\varepsilon}$ and boundedness of $A(u)$, we have
\begin{equation}\label{1.36}
\begin{split}
&\e\bigg|\int_0^t\langle  (A_{\varepsilon}(u(s))-A(u(s)))\nabla u^{\varepsilon}(s), \nabla \phi\rangle \dif s\bigg|\\
&\leq \|\nabla \phi\|_{L^{\infty}(\mt)}\e\int_0^t\int_{\mt}
|P_{\varepsilon}(A(u(s))(x)-A(u(s))(x)\|\nabla u^{\varepsilon}(s)(x)|\,\dif x\dif s\\
&\leq C \bigg(\e\int_0^t\|P_{\varepsilon}A(u(s))-A(u(s))\|_H^2\dif s\bigg)^{\frac{1}{2}}
\bigg(\int_0^t\e\|u^{\varepsilon}(s)\|_{H^{1}}^2\dif s\bigg)^{\frac{1}{2}}\rightarrow 0.
\end{split}
\end{equation}
By the weak convergence of $u^{\varepsilon}$,
\begin{equation}\label{1.37}
\begin{split}
&\e\int_0^t\langle  A(u(s))(\nabla u^{\varepsilon}(s)-\nabla u(s)), \nabla \phi\rangle \dif s\rightarrow 0.
\end{split}
\end{equation}
Putting together (\ref{1.34})--(\ref{1.37}) we arrive at
\begin{equation*}\label{1.32}
\begin{split}
&\langle u(t),\phi\rangle -\langle u_0,\phi\rangle -\int_0^t\langle  B(u(s)), \nabla \phi\rangle \dif s\\
&=-\int_0^t\langle  A(u(s))\nabla u(s), \nabla \phi\rangle \dif s+\int_0^t\langle \sigma(u(s))\,\dif W(s), \phi\rangle
\end{split}
\end{equation*}
proving the existence of a solution.
\vskip 0.4cm
Next we prove the uniqueness. Let $1>a_1>a_2>\cdots >a_n\cdots >0$ be a fixed sequence of  decreasing
positive numbers such that
\begin{equation*}\label{1.35a}
\int_{a_1}^1\frac{1}{u}\,\dif u=1, \cdots, \int_{a_n}^{a_{n-1}}\frac{1}{u}\,\dif u=n,\cdots
\end{equation*}
Let $\psi_n(u)$ be a continuous function such that
$\supt\psi_n\subset (a_n, a_{n-1})$ and
\begin{equation*}\label{1.36a}
0\leq \psi_n(u)\leq 2 \frac{1}{n}\times \frac{1}{u}, \quad \int_{a_n}^{a_{n-1}}\psi_n(u)\,\dif u=1.
\end{equation*}
Define
\begin{equation*}\label{1.37a}
\phi_n(x)=\int_0^{|x|}\int_0^y \psi_n(u)\,\dif u\dif y.
\end{equation*}
We have
\begin{equation}\label{1.38}
|\phi_n^{\prime}(x)|\leq 1, \quad \quad 0\leq \phi_n^{\prime\prime}(x)\leq 2\frac{1}{n}\times \frac{1}{|x|}.
\end{equation}
Introduce a functional $\Phi_n: H\rightarrow \mr$ by
$$\Phi_n(u)=\int_{\mt}\phi_n(u(z))\,\dif z, \quad u\in H.$$
Then we have
\begin{equation*}\label{1.39}
\Phi_n^{\prime}(u)(h)=\int_{\mt}\phi_n^{\prime}(u(z))h(z)\,\dif z,
\end{equation*}
and
\begin{equation}\label{1.40}
\Phi_n^{\prime\prime}(u)(h,g)=\int_{\mt}\phi_n^{\prime\prime}(u(z))h(z)g(z)\,\dif z.
\end{equation}
Suppose that $u_1, u_2$ are two solutions to equation (\ref{00.2}). We may apply the generalized It\^o formula \cite[Proposition A.1]{DHV} to deduce
\begin{equation}\label{1.41}
\begin{split}
&\Phi_n(u_1(t)-u_2(t))=\int_{\mt}\phi_n(u_1(t,z)-u_2(t,z))\,\dif z\\
&=\int_0^t \Phi_n^{\prime}(u_1(s)-u_2(s))(-\diver(B(u_1(s)))+\diver(B(u_2(s))))\dif s \\
&\quad+\int_0^t \Phi_n^{\prime}(u_1(s)-u_2(s))(\diver(A(u_1(s))\nabla u_1(s))-\diver(A(u_2(s))\nabla u_2(s)))\dif s \\
&\quad+\int_0^t \Phi_n^{\prime}(u_1(s)-u_2(s))(\sigma(u_1(s))-\sigma(u_1(s)))\,\dif W(s)\\
&\quad+ \frac{1}{2}\int_0^t \tr[ (\sigma(u_1(s))-\sigma(u_1(s)))^*\circ \Phi_n^{\prime\prime}(u_1(s)-u_2(s))\circ (\sigma(u_1(s))-\sigma(u_1(s)))]\dif s \\
&=I_n^1(t)+ I_n^2(t)+I_n^3(t)+I_n^4(t).
\end{split}
\end{equation}
We will bound each of the terms on the right. Now
\begin{equation}\label{1.42}
\begin{split}
I_n^1(t)
&=\int_0^t \int_{\mt}\phi_n^{\prime\prime}(u_1(s,z)-u_2(s,z))\nabla (u_1(s,z)-u_2(s,z))\cdot B(u_1(s,z))-B(u_2(s,z))\,\dif z\dif s\\
&\leq \frac{C}{n}\int_0^t \int_{\mt}|\nabla (u_1(s,z)-u_2(s,z))|\,\dif z\dif s\\
&\leq \frac{C}{n}\int_0^t \int_{\mt}|\nabla u_1(s,z)|\,\dif z\dif s+\frac{C}{n}\int_0^t \int_{\mt}|\nabla u_2(s,z)|\,\dif z\dif s\\
\end{split}
\end{equation}
where the Lipschitz continuity of $B$ and (\ref{1.38}) have been used. For $I_n^2$, we have
\begin{equation}\label{1.43}
\begin{split}
I_n^2(t)
&=-\int_0^t \int_{\mt}\phi_n^{\prime\prime}(u_1(s,z)-u_2(s,z)) \nabla (u_1(s,z)-u_2(s,z))\\
 &\quad\quad\quad\quad \cdot A(u_1(s,z))\nabla (u_1(s,z)-u_2(s,z)) \,\dif z\dif s\\
&\quad-\int_0^t \int_{\mt}\phi_n^{\prime\prime}(u_1(s,z)-u_2(s,z)) \nabla (u_1(s,z)-u_2(s,z))\\
 &\quad\quad\quad\quad  \cdot(A(u_1(s,z))-A(u_2(s,z)))\nabla u_2(s,z) \,\dif z\dif s\\
&\leq -\delta\int_0^t \int_{\mt}\phi_n^{\prime\prime}(u_1(s,z)-u_2(s,z))|\nabla (u_1(s,z)-u_2(s,z))|^2\,\dif z\dif s\\
&\quad+\frac{C}{n}\int_0^t \int_{\mt}|\nabla (u_1(s,z)-u_2(s,z))|\,|\nabla u_2(s,z)|\,\dif z\dif s\\
&\leq -\delta\int_0^t \int_{\mt}\phi_n^{\prime\prime}(u_1(s,z)-u_2(s,z))|\nabla (u_1(s,z)-u_2(s,z))|^2\,\dif z\dif s\\
&\quad+\frac{C}{n}\int_0^t\int_{\mt}|\nabla u_2(s,z)|^2\,\dif z\dif s+\frac{C}{n}\int_0^t\int_{\mt}|\nabla u_1(s,z)|^2\,\dif z\dif s,
\end{split}
\end{equation}
where the Lipschitz continuity of $A$ and (\ref{1.38}) have been used. The fourth term in (\ref{1.40}) can be estimated as follows.
\begin{equation}\label{1.44}
\begin{split}
I_n^4(t)&=\sum_{k=1}^{\infty}\int_0^t \langle \Phi_n^{\prime\prime}(u_1(s)-u_2(s))\circ (\sigma(u_1(s))-\sigma(u_1(s)))\bar{e}_k, (\sigma(u_1(s))-\sigma(u_1(s)))\bar{e}_k\rangle \dif s\\
&=\sum_{k=1}^{\infty}\int_0^t \int_{\mt}\phi_n^{\prime\prime}(u_1(s,z)-u_2(s,z))
|(\sigma(u_1(s))-\sigma(u_1(s)))\bar{e}_k(z)|^2\dif z\dif s\\
&=\int_0^t \int_{\mt}\phi_n^{\prime\prime}(u_1(s,z)-u_2(s,z))
\sum_{k=1}^{\infty}|(\sigma_k(u_1(s))-\sigma_k(u_1(s)))\bar{e}_k(z)|^2\dif z\dif s\\
&\leq \frac{C}{n}\int_0^t \int_{\mt}|u_1(s,z)-u_2(s,z)|\,\dif z\dif s.
\end{split}
\end{equation}
Substituting (\ref{1.42}), (\ref{1.43}) and (\ref{1.44}) into (\ref{1.41}) we get
\begin{equation*}
\begin{split}
\e\Phi_n(u_1(t)-u_2(t))&=\e\int_{\mt}\phi_n(u_1(t,z)-u_2(t,z))\,\dif z\\
&\leq \frac{C}{n}\bigg(\e\int_0^t \int_{\mt}|\nabla u_1(s,z)|^2\,\dif z\dif s+\e\int_0^t \int_{\mt}|\nabla u_2(s,z)|^2\,\dif z\dif s\bigg)\\
&\quad+\frac{C}{n}\e\int_0^t \int_{\mt}|u_1(s,z)-u_2(s,z)|\,\dif z\dif s
\end{split}
\end{equation*}
Hence by \eqref{1.15},
\begin{equation*}
\begin{split}
\e\int_{\mt}\phi_n(u_1(t,z)-u_2(t,z))\,\dif z&\leq \frac{C}{n}.
\end{split}
\end{equation*}
Letting $n\rightarrow \infty$ we obtain
$$\E\int_{\mt}|u_1(t,z)-u_2(t,z)|\,\dif z=0$$
This completes the proof of the theorem.
\end{proof}
\vskip 0.4cm
\noindent{\bf Acknowledgement.} We thank Elton Hsu for helpful discussions.

 \end{document}